\documentclass[12pt]{article}
\usepackage[utf8]{inputenc}

\usepackage{amsfonts,amsmath,amssymb,amsthm}
\usepackage{authblk}

\usepackage{graphicx}

\title{Convex polygons and separation of convex sets\thanks{Partially supported by Conacyt, M\'exico.}}
\author [1] {Eduardo Rivera-Campo}
\author [2] {Jorge Urrutia}
\affil [1] {Departamento de Matem\'aticas\\Universidad Aut\'onoma Metropolitana - Iztapalapa\\ \text{erc@xanum.uam.mx}}
\affil [2] {Instituto de Matem\'aticas\\Universidad Nacional Aut\'onoma de M\'exico\\
\text{urrutia@matem.uam.mx}}

\date{}

\newtheorem{lemma}{Lemma}
\newtheorem{theorem}{Theorem}

\begin{document}

\maketitle

\begin{abstract}
We prove that for any collection $F$  of $n \geq 2$ pairwise disjoint compact convex sets in the plane there is a pair of sets $A$ and $B$ in $F$ such that any line that separates $A$ from $B$ separates either $A$ or $B$ from a subcollection of $F$ with at least  $n/18$ sets.\\

    
\end{abstract}

\section{Introduction}

H. Tverberg \cite{T} proved that for each positive integer $k$, there is a minimum integer $f(k)$ such that for every collection $F$ of $f(k)$ or more plane compact convex sets with pairwise disjoint interiors, there is a line that separates one set in $F$  from a subcollection of $F$ with at least $k$ sets. R. Hope and M. Katchalski \cite{HK} showed that $3k+1 \leq f(k) \leq 12k-1$.

Later E. Rivera-Campo and J. T\"or\"ocsik \cite{RT} proved that in any collection  $F$ of $n\geq5$ pairwise disjoint compact convex sets in the plane, there is a pair of sets $A$ and $B$ such that any line that separates $A$ from $B$ separates either $A$ or $B$ from at least $\frac{n+28}{30}$ sets in $F$. In this paper we give a higher lower bound of $\frac{n}{18}$ sets in $F$ for any collection $F$ of $n\geq 2$ sets.

\section{Preliminary results}

H. Edelsbrunner \emph{et al} \cite{ERS} proved the following theorem, see L. Fejes-Toth \cite{F} for a related result. 

\begin{theorem}
\label{TeoEdelsbrunner}
Any collection of $n \geq 3$ compact, convex and pairwise disjoint sets in the plane may be covered with non-overlapping convex polygons with a total of not more than $6n-9$ sides. Furthermore no more than $3n-6$ distinct slopes are required.
\end{theorem}

We adapt part of the proof given in \cite{ERS} to obtain the following result.

\begin{lemma}
\label{lema}
Let $\mathcal{P} = \{P_1, P_2, \ldots, P_n\}$ be a collection of $n \geq 3$ pairwise disjoint convex polygons in the plane. There exists a collection $\mathcal{R} = \{R_1, R_2, \ldots, R_n\}$ of pairwise non-overlapping convex polygons such that:
\begin{enumerate}
    \item For $i=1, 2, \ldots, n$, $P_i$ is contained in $R_i$.
    \item For $i=1, 2, \ldots, n$, each side of $R_i$ supports a side of $P_i$.
    \item The total number of sides among polygons $R_1, R_2, \ldots, R_n$ is at most $9n-9$.
\end{enumerate}
\end{lemma}

\begin{proof}

A side $s$ of a polygon $P_i \in \mathcal P$ is \emph{reducible with respect to $\mathcal{P}$} if the triangle $t_s$ (not containing $P_i$), bounded by $s$ and the lines supporting the sides of $P_i$ incident to $s$, does not intersect the interior of any another polygon $P_j$. Equivalently, a side $s$ of a polygon $P_i \in \mathcal{P}$ is reducible with respect to $\mathcal{P}$ if it vanishes before reaching the interior of a polygon $P_j$ when it is translated in the direction orthogonal to $s$ and away from $P_i$ (see Fig. \ref{ReducibleSide}).

\begin{figure}[h!]
\centering
  \includegraphics[width= 1.7in]{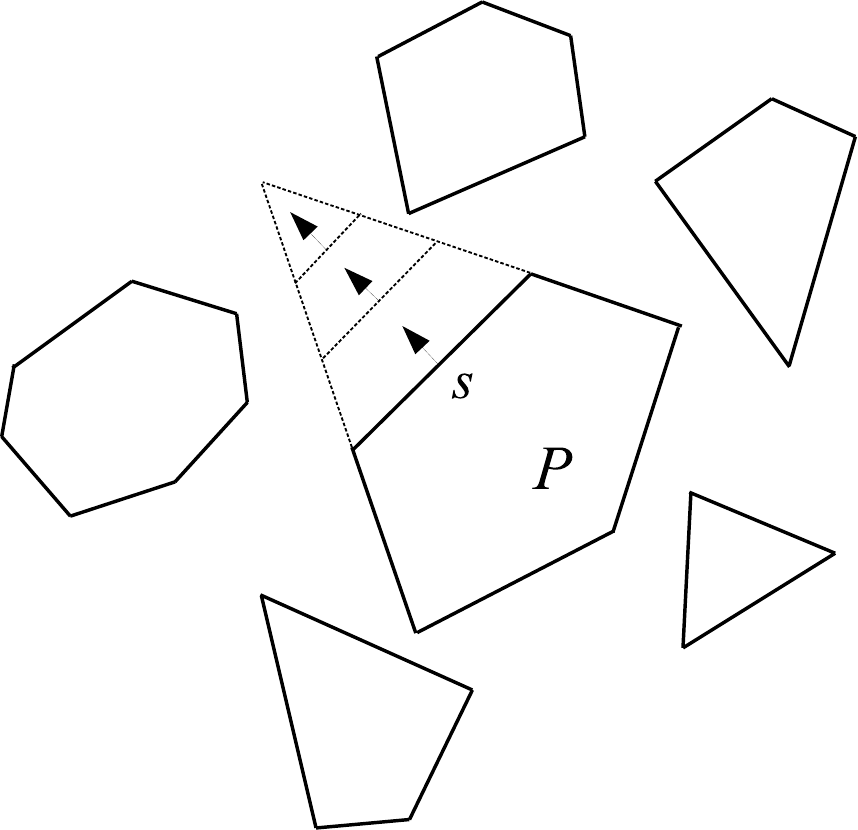}
  \caption{Polygon $P$ with a reducible side $s$.}
  \label{ReducibleSide}
\end{figure}

We modify $\mathcal{P}$ by growing, one at a time, polygons in $\mathcal{P}$  as follows: if some polygon $P_i \in \mathcal{P}$ has a reducible side $s$ with respect to $\mathcal{P}$ substitute $P_i$ in $\mathcal{P}$ with polygon $P_i \cup t_s $. Repeat this until no polygon in $\mathcal{P}$ has a reducible side. Denote by $\mathcal{R} = \{R_1, R_2, \ldots, R_n\}$ the family of polygons thus obtained. 

Its is clear that $\mathcal{R}$ satisfies conditions 1) and 2), we claim it also satisfies condition 3. To prove this consider a third family $\mathcal{Q} = \{Q_1, Q_2, \ldots, Q_n\}$ of convex polygons obtained from $\mathcal{R}$ by further growing polygons $R_1, R_2, \ldots, R_n$ in the following way.

A side $s$ of a polygon $R_i$ in $\mathcal{R}$ is \emph{free} with respect to $\mathcal{R}$ if the open polygonal region, not containing $R_i$, which is determined by $s$ and the lines supporting the two sides of $R_i$ adjacent to $s$ contains no points of polygons in $\mathcal{R}$. 

Let $C$ be a circle containing all polygons in $\mathcal{R}$. Starting with the non free sides of polygons $R_i \in \mathcal{R}$, translate one at a time a side $s$ of a polygon $R_i \in \mathcal{R}$ (also in the direction orthogonal to $s$ and away from $R_i$) as far as possible without intersecting the interior of another (possibly already grown) polygon $R_j$ or the exterior of $C$ (see Fig. \ref{translate}).

\begin{figure}[h!]
\centering
  \includegraphics[width= 3.5in]{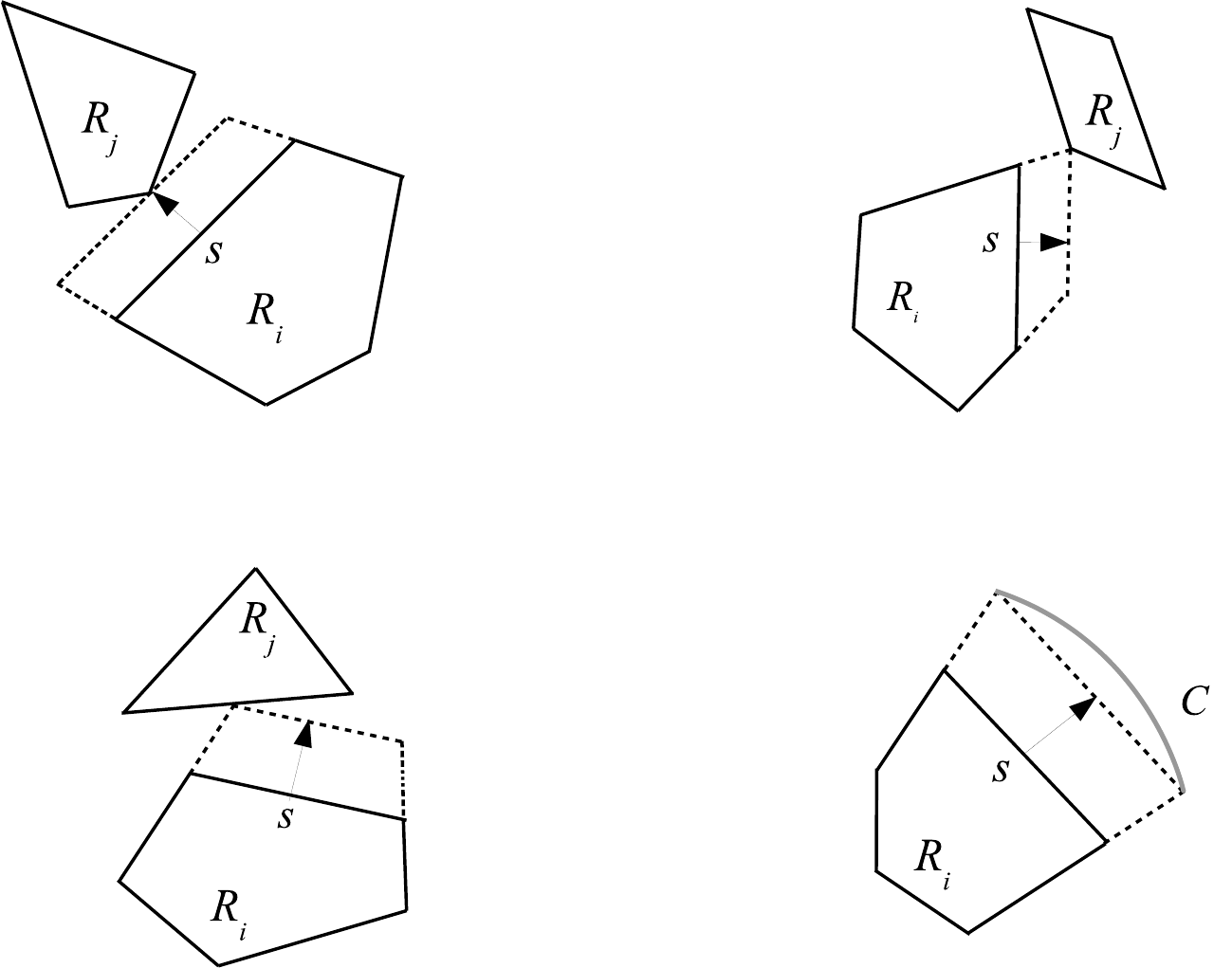}
  \caption{Side $s$ of $R_i$ reaches a polygon $R_j$ or reaches circle $C$.}
  \label{translate}
\end{figure}

Stop when no side can be translated as indicated and let $\mathcal{Q} = \{Q_1, Q_2, \ldots, \allowbreak Q_n\}$ where, for $i=1, 2, \ldots, n$, $Q_i$ is the polygon obtained from $R_i$ in this manner. We claim that if $C$ is chosen large enough, then each polygon in $\mathcal{Q}$ contains at most 3 free sides and all non free sides of a polygon $Q_i$ in $\mathcal{Q}$ are in contact with a polygon $Q_j$ with $i \neq j$.

By the choice of $\mathcal{R}$ no side of a polygon $R_j$ vanishes, therefore for $i= 1, 2, \ldots, n$ the number of sides of $R_i$ is equal to the number of sides of $Q_i$. 

We define a plane graph  $G$ with one vertex $v_i$ placed in the interior of each polygon $Q_i \in \mathcal{Q}$ and  a vertex $v_{n+1}$ placed outside circle $C$. The edge set of $G$ consists of 6 types of edges as follows:

a) If a side of a polygon $Q_i$ and a side of a polygon $Q_j$ have a segment $l$ in common choose an arbitrary point $p$ in $l$ and draw an edge $v_iv_j$ through $p$ as in Fig. \ref{a)b)c)} (left).

b) If a vertex $p$ of a polygon $Q_i$ lies in the interior of a side of a polygon $Q_j$  and $p$ is not a vertex of a polygon in $\mathcal{Q}$ other than $Q_i$, then draw an edge $v_iv_j$ through $p$ as in Fig.  \ref{a)b)c)} (center).

c) If two or more polygons $Q_{i_1}, Q_{i_2}, \ldots, Q_{i_r}$ share a vertex $p$ which lies in the interior of a side of a polygon $Q_j$ and $p$ is not a point of any polygon $Q_k$ other than $Q_{i_1}, Q_{i_2}, \ldots, Q_{i_r}$, then there is a neighbourhood $N$ of $p$ containing no points of polygons $Q_k$ with $k \neq j,i_1, i_2, \ldots, i_r$. In this case draw a cycle with edges $v_jv_{i_1}, v_{i_1}v_{i_2}, v_{i_2}v_{i_3}, \ldots, v_{i_{r-1}}v_{i_r}, v_{i_r}v_j$ as in Fig.  \ref{a)b)c)} (right).

\begin{figure}[h!]
\centering
  \includegraphics[width= 4in]{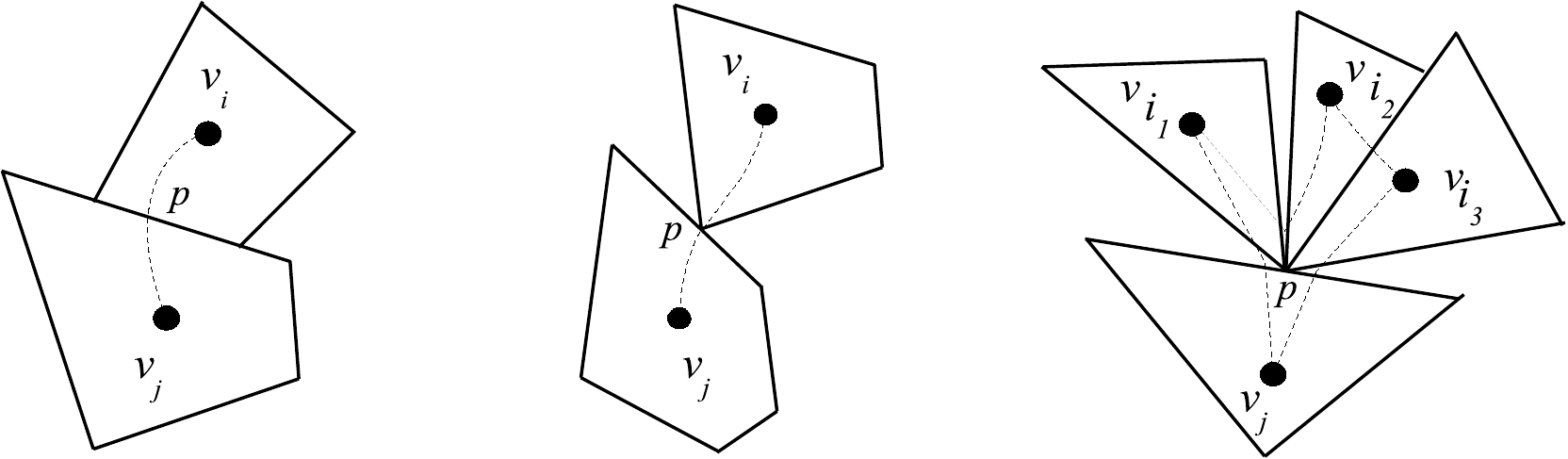}
  \caption{Edge type a) (left), edge type b) (center) and edges type c) (right).}
  \label{a)b)c)}
\end{figure}

d) If three or more polygons $Q_{i_1}, Q_{i_2}, \ldots, Q_{i_r}$ share a vertex $p$ and $p$ is not a point of any polygon $Q_k$ other than $Q_{i_1}, Q_{i_2}, \ldots, Q_{i_r}$, then there is a neighbourhood $N$ of $p$ containing no points of any polygon $Q_k$ with $k \neq i_1, i_2, \ldots, i_r$. In this case draw a cycle with edges $v_{i_1}v_{i_2}, v_{i_2}v_{i_3}, \ldots, v_{i_{r-1}}v_{i_r}, v_{i_r}v_1$ as in Fig. \ref{d)e} (left).

e) If two polygons $Q_i$ and $Q_j$ share a vertex $p$ and $p$ is not a point of any polygon in $\mathcal{R}$ other than $Q_i$ and $Q_j$, then draw an edge $v_iv_j$ through $p$ as in Fig. \ref{d)e} (right).

\begin{figure}[h!]
\centering
  \includegraphics[width= 3in]{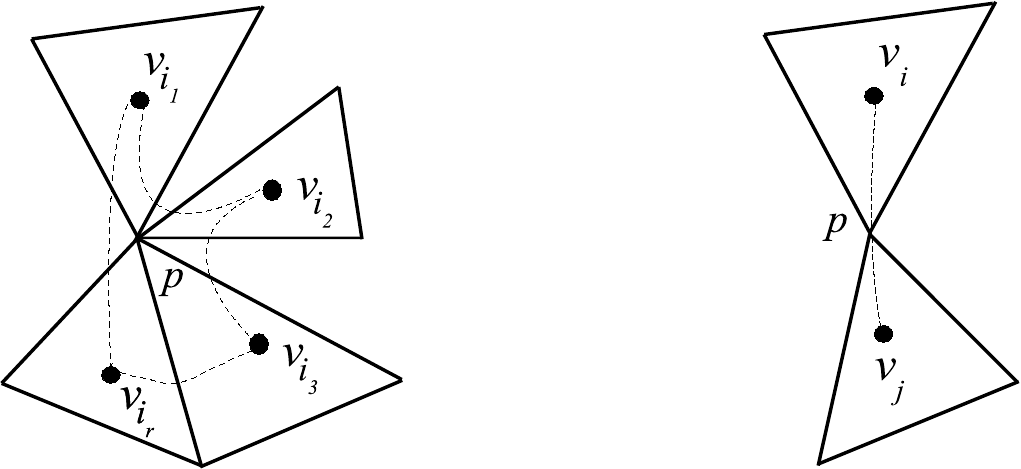}
  \caption{Edges type d) (left) and edge type e) (right).}
  \label{d)e}
\end{figure}

f) For each free side $s$ of a polygon $Q_i$ add an edge  $v_iv_{n+1}$ drawn through an interior of point $p$ in $s$. 

Notice that edges of type f) are the only possible multiple edges of $G$. Therefore $G$ is a plane multigraph with $n+1$ vertices and $m \leq (3(n+1) - 6) + (x_2 + 2x_3) = (3n-3) + (x_2 + 2x_3)$ edges, where for $i = 2, 3$, $x_i$ denotes the number of polygons in $\mathcal{Q}$ with $i$ free sides.

\begin{figure}[h!]
\centering
  \includegraphics[width= 4in]{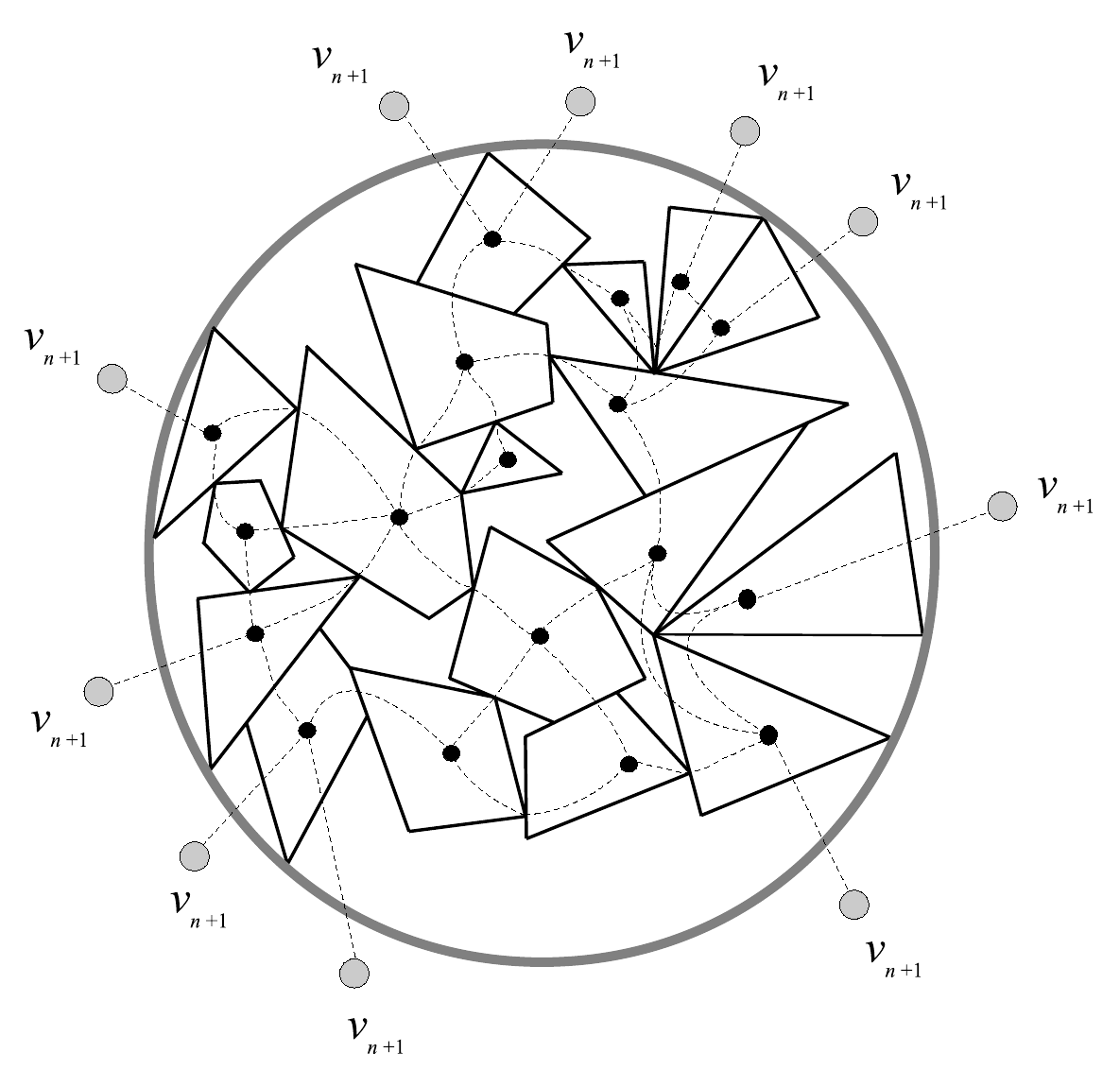}
  \caption{Graph $G$ with $n = 18$, $x_1 = 6$, $x_2 = 2$, $m_e = 1$ and $m_f = 10$.}
  \label{Q + C}
\end{figure}

We say that an edge $e= v_iv_j$ of $G$ intersects a side $s$ of a polygon in $\mathcal{Q}$ if $e$ is drawn through either an interior point of $s$ or through a vertex of $s$. Each non free side among polygons in $\mathcal{Q}$ is intersected by at least one edge of $G$ of type a), b), c), d) or e) and each free side is intersected by an edge of type f). We claim that if an edge  $v_iv_j$ of $G$ is of type e) then two of the 4 sides intersected by $v_iv_j$ are  intersected by at least another edge of $G$. 

\bigskip

\noindent \emph{Proof of Claim}. Let $v_iv_j$ be an edge of $G$ of type e). Then either two sides $s, t$ of $Q_i$ (of $Q_j)$ or one side $s$ of $Q_i$ and one side $t$ of $Q_j$ are such that the lines supporting $s$ and $t$ separate $Q_i$ and $Q_j$ (see Fig. \ref{separating}). 

\begin{figure}[h!]
\centering
  \includegraphics[width= 3in]{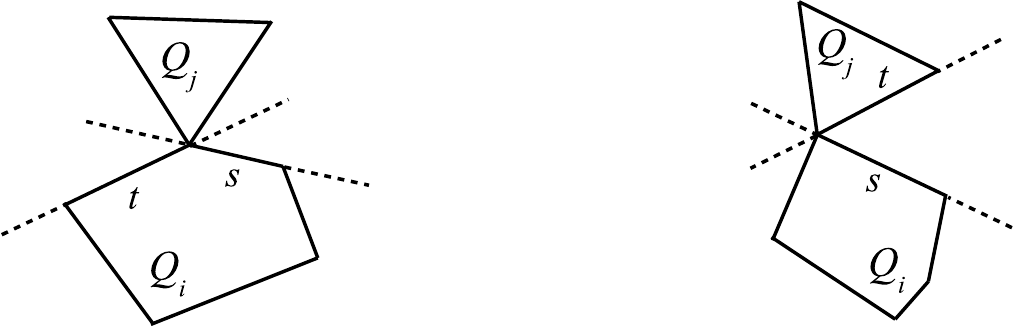}
  \caption{Lines supporting $s$ and $t$ separate $Q_i$ and $Q_j$.}
  \label{separating}
\end{figure}

Without loss of generality we assume the line supporting a side $s$ of $Q_i$ separates $Q_i$ and $Q_j$. Let $b$ be the other side of $Q_i$ incident in the common vertex $p$ of $Q_i$ and $Q_j$. Then side $b$ must contain a point of a polygon $Q_k$ other than $Q_i$ and $Q_j$, otherwise $b$ could be translated away from $Q_i$ in the direction orthogonal to $b$ which is not possible by the properties of $\mathcal{Q}$ (see Fig. \ref{ext}). This means that side $b$ is intersected by an edge of $G$ other than edge $v_iv_j$. 

\begin{figure}[h!]
\centering
  \includegraphics[width= 1.5in]{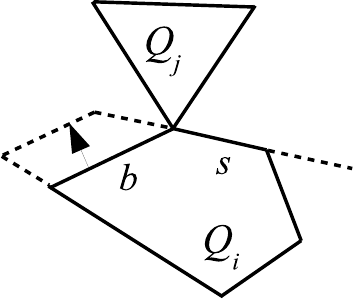}
  \caption{Side $b$ could be translated away from $Q_i$.}
  \label{ext}
\end{figure}

Analogously a side $c \neq b$, incident with $t$, is intersected by an edge of $G$ other than $v_iv_j$. \hfill $\square$

\bigskip

Let $T$ denote the total number of sides among polygons in $\mathcal{Q}$. Also let $m$, $m_e$ and $m_f$ denote the number of edges of $G$, the number of edges of $G$ of type e) and the number of edges of $G$ of type f), respectively. Since polygons in $\mathcal{Q}$ contain at most 3 free edges, $m_f = x_1 + 2x_2 + 3x_3 $, where for $i = 1, 2, 3$ $x_i$ is the number of polygons in $Q$ containing $i$ free edges.

In order to bound the number of sides among polygons in $\mathcal{Q}$ we add the number of sides intersected by edges of $G$ and subtract the number of sides intersected by at least two edges of $G$ (2 sides for each edge of type e)). 
Considering that edges of types a), b), c), and d), intersect at most 3 sides; edges of type e) intersect 4 sides and edges of type f) intersect one side, we have: 
\begin{align*}
    T &\leq (3(m - (m_e + m_f)) + 4m_e + m_f) - 2m_e\\
    &= 3m - m_e - 2m_f\\
    &\leq 3m - 2m_f\\
    &\leq 3((3n - 3) + (x_2 + 2x_3)) - 2(x_1 + 2x_2 + 3x_3)\\
    &= 9n - 9 - 2x_1 - x_2 \\
    &\leq 9n - 9.
    \end{align*}

\end{proof}

Consider the following familly of plane geometric graphs $G_t$, $t=1,2, \ldots $ given by  Edelsbrunner \emph{et al} \cite{ERS}  where it is used to show that the bounds in the number of sides and slopes on Theorem \ref{TeoEdelsbrunner} are tight. 

$G_1$ is the graph with $7$ vertices shown on Fig. \ref{G_t} (left); for $t \geq 1$, $G_{t+1}$ is the graph obtained from $G_t$ as in Fig. \ref{G_t} (right). For $k \geq 1$ the graph $G_k$ is a plane geometric graph with $n = 3k$ internal faces. The three inner-most faces are quadrilaterals while all other faces are hexagons. 

\begin{figure}[h!]
\centering
  \includegraphics[width= 3.5in]{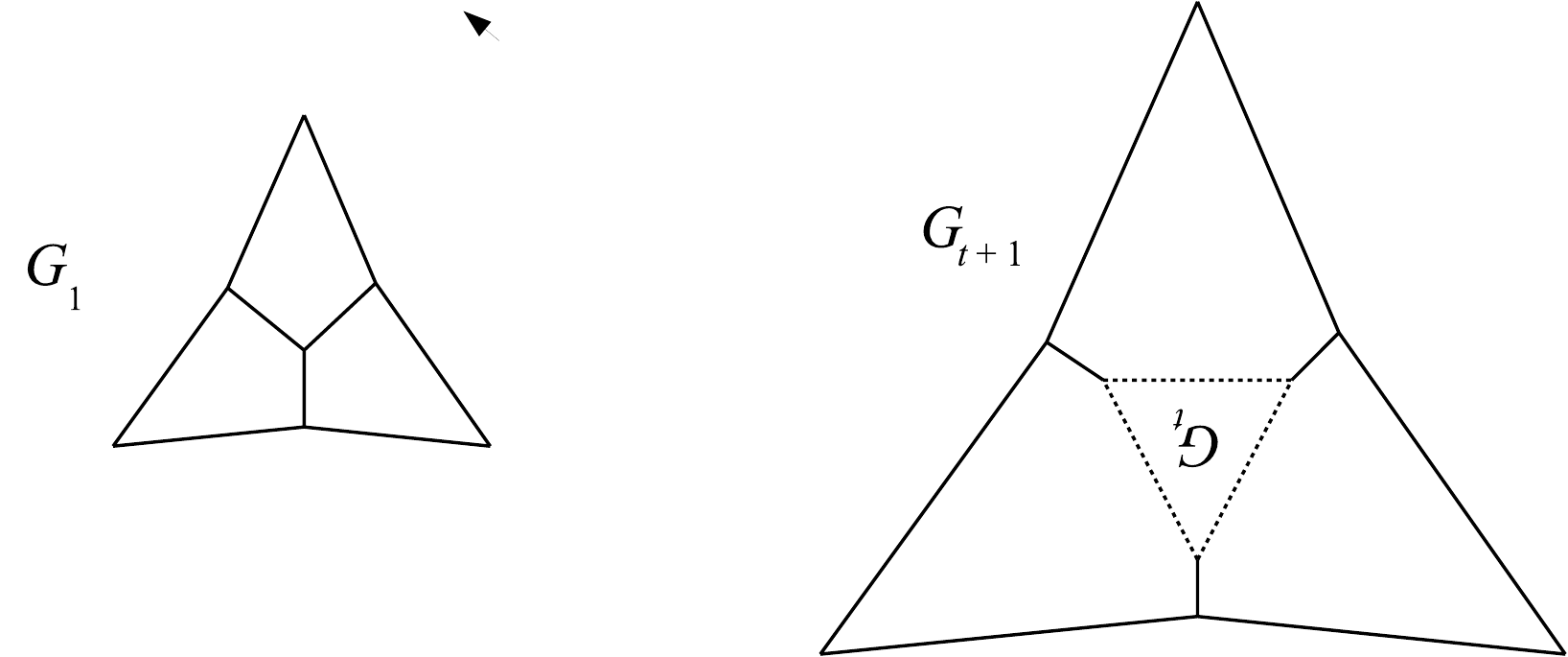}
  \caption{Left: Graph $G_1$. Right: Graph $G_{t+1}$ obtained from graph $G_t$ (placed upside down inside the dotted triangle) .}
  \label{G_t}
\end{figure}

Place an hexagon inside each of the three inner-most faces of $G_k$ as in Fig. \ref{HexagonOctagon} (left), an octagon inside each outer-most internal face of $G_k$ as in Fig. \ref{HexagonOctagon} (right) and a nonagon inside each other internal face of $G_k$ as in Fig. \ref{Nonagon}.

\begin{figure}[h!]
\centering
  \includegraphics[width= 3.5in]{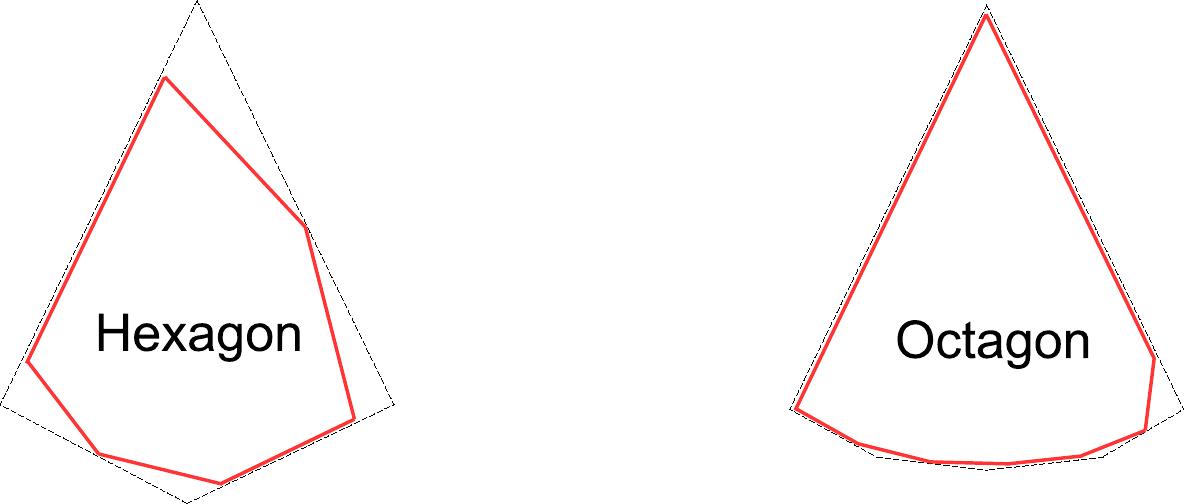}
  \caption{Hexagon inside inner-most face of $G_k$. Octagon inside outer-most internal face of $G_k$.}
  \label{HexagonOctagon}
\end{figure}

\begin{figure}[h!]
\centering
  \includegraphics[width= 1.5in]{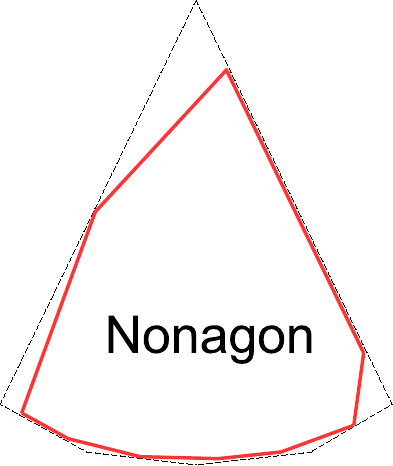}
  \caption{Nonagon inside internal face of $G_k$}
  \label{Nonagon}
\end{figure}

Since each face of $G_k$ has an even number of sides, the placement of the polygons described above may be done in such a way that for each internal edge $e$ of $G_k$ a side $s_e$ of the polygon $R_i$ lying on one side of $e$ is parallel (and closed enough) to $e$ while a vertex $v_e$ of the polygon $R_j$ lying on the other side of $e$ lies on $e$ as in Fig. \ref{edge}.

\begin{figure}[h!]
\centering
  \includegraphics[width= 2.5in]{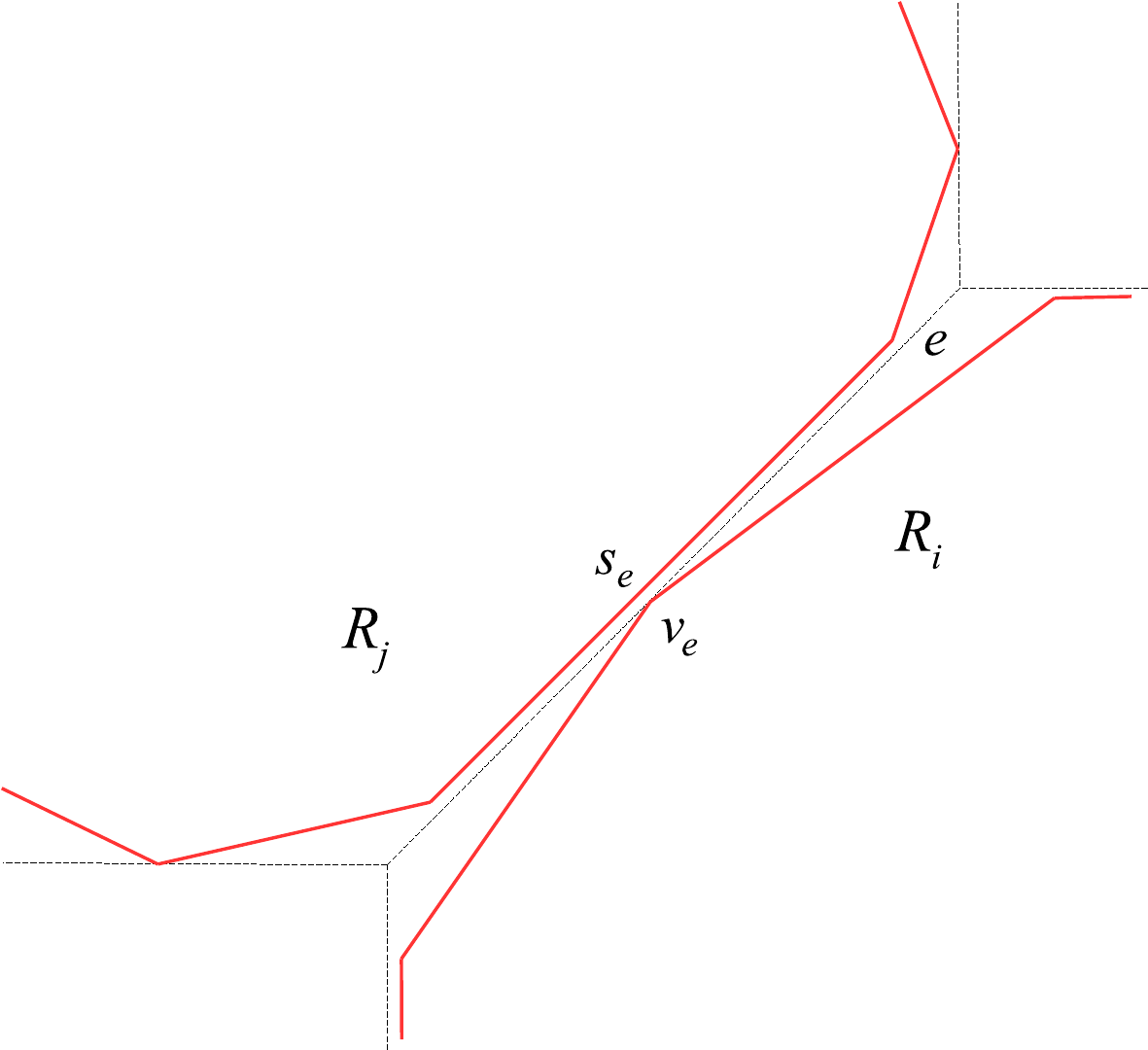}
  \caption{Polygons on both sides of an internal edge of $G_k$.}
  \label{edge}
\end{figure}

Let $\mathcal{R} = \{R_1, R_2, \ldots R_n\}$ be the set of polygons described above. The total number of sides among polygons in $\mathcal{R}$ is 
$$m=3(6) + 3(8) + (n-6)(9) = 9n-12$$
Each side among polygons in $\mathcal{R}$ is relevant with respect to $\mathcal{R}$, therefore our bound on the number of sides in Lemma \ref{lema} is almost tight.

\section{Main Results}

In this section we present our main results.

\begin{lemma}
\label{lemaprincipal}

Let $\mathcal{P} = \{P_1, P_2, \ldots, P_n\}$ be a collection of $n \geq 3$ pairwise disjoint convex polygons in the plane. There is a side $s$ of a polygon  $P_i$ such that the line supporting $s$ separates $P_i$ from a subcollection $\mathcal{F}$ of $\mathcal{P}$ with at least $\frac{n}{18}$ polygons.
\end{lemma}

\begin{proof}
Let $\mathcal{R} = \{R_1, R_2, \ldots R_n\}$ be a collection of pairwise disjoint convex polygons satisfying conditions 1), 2) and 3) in Lemma \ref{lema} and let $\mathcal{L} = \{l_1, l_2, \ldots, l_m\}$ be the (multi)set of lines supporting the sides of each polygon in $\mathcal{R}$. We include $c$ copies of a line $l$ if $l$ supports sides of $c$ polygons in $\mathcal{R}$. Therefore we can associate a unique polygon $R_{i(k)}$
to each line $l_k \in \mathcal{L}$ such that a side of $R_{i(k)}$ is supported by $l_k$. For $k=1, 2, \ldots, m$ let $H_k^-$ be the closed halfplane determined by $l_k$ that does not contain $R_{i(k)}$.

Define a bipartite graph $F$ with a vertex $v_j$ for each polygon $R_j \in \mathcal{R}$ and a vertex $w_k$ for each line $l_k \in \mathcal{L}$. Graph $F$ has an edge $v_jw_k$ if polygon $R_j$ is contained in $H_k^-$.

For each pair of polygons  $R_i, R_j$ there is at least one side $s$ of one of them such that the line supporting $s$ separates $R_i$ from $R_j$. Therefore graph $F$ has at least one edge for each pair $\{i,j\}$ with $1 \leq i < j \leq n$. This implies that there is a vertex $w_k$ whose degree in $F$ is at least $\binom{n}{2}/m \geq \binom{n}{2}/(9n-9) = n/18
$.

This means that the line $l_k$ separates polygon $R_{i(k)}$ from at least $n/18$ polygons in $\mathcal{R}$. By Property 2) in Lemma \ref{lema}, line $l_k$ supports a side of polygon $P_{i(k)}$.

\end{proof}

\begin{theorem}

\label{teoremaprincipal}
In any collection $\mathcal{C}$ of $n \geq 2$  pairwise disjoint compact convex sets in the plane, there is a pair of sets $A$ and $B$ such that any line that separates $A$ from $B$ separates either $A$ or $B$ from a subcollection $\mathcal{F}$ of $\mathcal{C}$ with at least $\frac{n}{18}$ sets. 
\end{theorem}

\begin{proof}
Let $n \geq 2$ be an integer and $\mathcal{C} = \{C_1, C_2, \ldots, C_n\}$ be a collection of pairwise disjoint compact convex sets in the plane and let $T$ be a triangle containing all sets in $\mathcal{C}$ in its interior.

For any line $l$ let $t_l^-$ and $t_l^+$ be the number of sets in $\mathcal{C}$ lying to the left and to the right of $l$, respectively; for a horizontal line $l$, $t_l^-$ and $t_l^+$ are the number of sets in $\mathcal{C}$ lying above and bellow $l$, respectively. For $1 \leq i < j \leq n$ let $l_{ij}$ be a line separating $C_i$ from $C_j$ for which $max\{t_{l_{ij}}^-, t_{l_{ij}}^+\}$ is as small as posible. 

Each line $l_{ij}$ determines two closed halfplanes $H_{ij}$ and $H_{ji}$ containing $C_i$ and $C_j$, respectively. For $i= 1, 2, \ldots, n$, let $P_i = T \cap ( {\bigcap}_{j\neq i}^n H_{ij})$. Then For $i= 1, 2, \ldots, n$, $P_i$ is a convex polygon that contains set $C_i$ such that each side is supported by a line $l_{ij}$ or by a side of $T$. 

Let $\mathcal{P}= \{P_1, P_2, \ldots, P_n\}$. By Lemma \ref{lemaprincipal}, there is a side $s$ of a polygon $P_i$ such that the line $l(s)$ supporting $s$ separates $P_i$ from a subcollection of $\mathcal{P}$ with at least $\frac{n}{18}$ polygons. 

Since no side of $T$ separates polygons in $\mathcal{P}$, $l(s)$ is of the form $l_{ij}$. By the choice of $l_{ij}$ each line that separates sets $C_i$ from $C_j$ separates either $C_i$ or $C_j$ from at least $\frac{n}{18}$ sets in $\mathcal{C}$.
\end{proof}

\end{document}